\documentclass[reqno,12pt,a4paper]{amsart}
\usepackage{amsfonts,euscript,amsmath,amssymb,amsthm,eufrak}
\usepackage[english]{babel}
\usepackage[matrix,arrow,curve]{xy}

\newcommand{\ol}{\overline}

\DeclareMathOperator{\Ker}{Ker}
\DeclareMathOperator{\Ran}{Ran}\DeclareMathOperator*{\res}{res}
 \DeclareMathOperator*{\im}{Im}
 \DeclareMathOperator*{\rank}{rank}
 \DeclareMathOperator*{\dist}{dist}

\newcommand{\Z}{\mathbb Z}

\newcommand{\C}{\mathbb C}
\newcommand{\E}{{\it{\EuScript E}}}
\newcommand{\Pm}{{\it{\EuScript P}}}

\newcommand{\kratno}{\lower.5ex\hbox{$\,\vdots\,$}}

\renewcommand{\leq}{\leqslant}
\renewcommand{\geq}{\geqslant}
\renewcommand{\dots}{\ldots}

\renewcommand{\theequation}{\arabic{section}.\arabic{equation}}

\theoremstyle{plain}
\newtheorem{thm}{Theorem}[section]
\newtheorem{lem}[thm]{Lemma}

\newtheorem{prop}[thm]{Proposition}
\theoremstyle{remark}
\newtheorem{rem}[thm]{Remark}
\theoremstyle{definition}
\newtheorem{dfn}[thm]{Definition}

\newcommand{\qForm}[8]{\begin{pmatrix} #1 & #2 \end{pmatrix}\!\begin{pmatrix}
#3 & #4\\ #5 & #6 \end{pmatrix}\!
\begin{pmatrix} #7\vphantom{#3 #4} \\ #8\vphantom{#5 #6} \end{pmatrix} }

\newcommand{\Fm}{F_-}
\newcommand{\Fp}{F_+}
\newcommand{\Fpm}{F_\pm}

\newcommand{\Bcondm}{\Gamma_-}
\newcommand{\Bcondp}{\Gamma_+}
\newcommand{\Bcondmd}{\Gamma_-^\perp}
\newcommand{\Bcondpd}{\Gamma_+^\perp}
\newcommand{\wt}{\widetilde}
\newcommand{\Tp}{T_+}
\newcommand{\Tm}{T_-}
\newcommand{\Em}{E_-} \newcommand{\PEm}{P_-}
\newcommand{\Ep}{E_+} \newcommand{\PEp}{P_+}
 \newcommand{\PEpm}{P_\pm}

\input{UTPackage.tex}

\begin{document}

\title[Vector-valued Sturm-Liouville problem. I. Uniqueness theorem]{Inverse vector-valued Sturm-Liouville problem.\\ I. Uniqueness theorem}

\author[Dmitry Chelkak]{Dmitry Chelkak$^\mathrm{a,c}$}
\thanks{}
\email{}

\author[Sergey Matveenko]{Sergey Matveenko$^\mathrm{b,c}$}

\thanks{\textsc{${}^\mathrm{A}$ St.Petersburg Department of Steklov Mathematical Institute (PDMI RAS).
Fontanka~27, 191023 St.Petersburg, Russia.}}

\thanks{\textsc{${}^\mathrm{B}$ Saint-Petersburg State University of Aerospace Instrumentation.
Bol'shaya Morskaya st.~67, 190000, St. Petersburg,  Russia.}}

\thanks{\textsc{${}^\mathrm{C}$ Chebyshev Laboratory, Department of Mathematics and
Mechanics, Saint-Petersburg State University. 14th Line, 29b, 199178 Saint-Petersburg,
Russia.}}

\thanks{{\it E-mail addresses:} \texttt{dchelkak@pdmi.ras.ru}, \texttt{matveis239@gmail.com}}

\begin{abstract}
This paper starts a series devoted to the vector-valued Sturm-Liouville problem
$-\psi''+V(x)\psi=\lambda\psi$, $\psi\in L^2([0,1];\C^N)$, with separated boundary conditions.
The overall goal of the series is to give a complete characterization of classes of spectral data corresponding to potentials $V=V^*\in L^p([0,1];\C^{N\times N})$ for a fixed $1\le p <+\infty$ and separated boundary conditions having the most general form. In the first paper we briefly describe our
approach to this inverse problem and prove some preliminary results including the relevant uniqueness theorem.
\end{abstract}

\maketitle

\section{Introduction}

\subsection{Introduction} Inverse spectral problems for scalar Sturm-Liouville operators were actively investigated after the seminal paper \cite{Borg}. In this paper Borg showed that a scalar potential $v(x)$ is uniquely
determined by the two spectra of Sturm-Liouville operators $-\psi''+v(x)\psi=\lambda\psi$, $x\in[0,1]$, with the same boundary
conditions at~$1$ and different boundary conditions at~$0$. Later on Marchenko \cite{Mar50} proved that the so-called spectral function (or, equivalently, the Weyl-Titchmarsh function $m(\lambda)$, $\lambda\in\C$, which is analytic outside of the spectrum and has simple poles at eigenvalues) determines the potential uniquely. Shortly afterwards Gelfand and Levitan \cite{GelLev} suggested an algorithm for the reconstruction of $v(x)$ starting with spectral data, and another approach to this inverse problem was developed by Krein \cite{Kr1,Kr4}. After several decades of intensive research this activity culminated when characterization theorems, i.e., complete descriptions of spectral data that correspond to potentials from a given ``nice'' class, appeared. More information about the inverse spectral theory for scalar Sturm-Liouville operators can be found in monographs \cite{M86, Lev84, PT, FYu01}, survey \cite{Ges07}, and references cited therein.

By contrast, similar characterization problems for vector-valued Sturm-Liouville operators were not explored as much (though vector-valued Shr\"odinger operators on the half-line were discussed in one of the first monographs on the subject \cite{AM63}), an interest in the corresponding detailed inverse theory has quickened only recently \cite{Mal05,Yu06a,Yu06b,CKIsoSets,CK,MT,Bo11,Bo12}, see also miscellaneous results in \cite{ChSh97, JL98a, JL98b, Car99, CGHL00, Sh01, Car02}.

In our paper we consider the inverse spectral problem for self-adjoint operators acting in $L^2([0,1];\C^N)$ which are defined by the differential expression
\begin{equation}\label{SLPi}
L\psi=-\psi''+V(x)\psi,
\end{equation}
and \emph{separated boundary conditions of general type}, where $V=V^*\in L^p([0,1];\C^{N\times N})$, $1\leq p<+\infty$, is an $N\times N$ matrix-valued potential. It is worth noting that in the existing literature only two special cases of boundary conditions are usually dealt with: either Dirichlet $\psi(0)=\psi(1)=0$ or Neumann-type (also known as Robin or third-type) $\psi'(0)-a\psi(0)=\psi'(1)+b\psi(1)=0$ ones. In some sense, these two cases resemble the scalar situation (e.g., the leading term in the asymptotics of eigenvalues is $\pi^2n^2$), while for the vector-valued problem there is an additional geometry of the space $\C^N$ that can come into play. Namely, if we consider most general separated boundary conditions
\begin{equation}
\label{BoundConditionIntro}
\begin{aligned}
\Bcondm\psi=T_-^\perp(\psi'(0)-a \psi(0))-T_-\psi(0)=0,\\
\Bcondp\psi=T_+^\perp(\psi'(1)+b\psi(1))-T_+\psi(1)=0,
\end{aligned}
\end{equation}
where $T_\pm$ and $T_\pm^\perp=I-T_\pm$ are orthogonal projectors acting in~$\C^N$, and
\[
a=T_-^\perp a T_-^\perp \colon H_-^\perp\to H_-^\perp,\quad b=T_+^\perp bT_+^\perp\colon H_+^\perp\to
H_+^\perp
\]
are self-adjoint  linear operators acting in $H_\pm^\perp=\Ran T_\pm^\perp=\Ker T_\pm\subset\C^N$, respectively, then even the leading terms in the asymptotics of eigenvalues depend on the mutual location of $H_-$ and $H_+$ in $\C^N$. Thus, our research is devoted to the inverse spectral problem for operators (\ref{SLPi})--(\ref{BoundConditionIntro}), especially to the case when all projectors $T_\pm$, $T_\pm^\perp$ are nontrivial and no special orthogonality relations hold.

The motivation for this work is two-fold. The first is to give a consistent solution to the spectral characterization problem for operators (\ref{SLPi})--(\ref{BoundConditionIntro}), the case which, to the best of our knowledge, has not been treated in the literature in full generality. The second is to demonstrate a new approach to characterization results in inverse spectral theory for 1D differential operators. A scheme of this kind was firstly introduced in \cite{PT} and further developed in \cite{CKK,Che09}. Finally, to prove a characterization theorem, we do not need neither (a) any explicit reconstruction procedure applied to given ``generic'' spectral data, thus we avoid a careful tracing of fine spectral data asymptotics along a reconstruction algorithm, nor (b) any differentiability-type arguments for the mapping
\[
\Lambda:\{\text{operators}\}\mapsto\{\text{spectral~data}\}.
\]
Note that, even if the mapping $\Lambda$ is actually smooth, it may be hard to compute and to analyze its gradient, especially if some regularization of spectral data is plugged in. Of course, both~(a) and~(b) are of independent interest. Moreover, in our approach we directly use a sort of~(a)
for ``nice perturbations'' of spectral data, so, by no means, a (formal) reconstruction algorithm plays a very important role in any inverse problem.
Nevertheless, we believe that there are many situations in which a robust proof of characterization results, not going into technicalities of a particular
 reconstruction procedure (convergence issues, smoothness of intermediary objects etc.), is appreciated.



\subsection{Inverse problem for vector-valued Sturm-Liouville operators. Overall strategy.}

In this Section we briefly describe our approach to the inverse spectral problem for differential operators (\ref{SLPi})--(\ref{BoundConditionIntro}), especially to characterization results for properly chosen spectral data of these operators. As the spectrum itself doesn't determine $V(x)$ uniquely besides specific setups similar to the classical Ambarzumayn theorem \cite{Amb}, one should supplement eigenvalues with some additional parameters. For scalar operators, there are two standard choices of those: normalizing constants $\alpha_n(v)$ originally introduced by Marchenko \cite{Mar50,M86}, and norming constants $\nu_n(v)$ used by Trubowitz and co-authors \cite{IT,IMT,DT,PT}. It is worth noting that, once the inverse problem is solved for one of the choices of additional parameters, this solution can be easily translated to the other choice (e.g., see~\cite[Sect.~3]{Che09}). For the scalar Sturm-Liouville problem on $[0,1]$ norming constants $\nu_n$ are slightly more useful as the condition $\nu_n(v)\equiv 0$ exactly characterizes symmetric potentials $v(x)\equiv v(1\!-\!x)$. Nevertheless, it is not clear what is a correct analogue of $\nu_n$'s for vector-valued potentials, thus, following \cite{Yu06a, Yu06b, CKIsoSets, CK, MT, Bo11, Bo12} and others, in our work the Weil-Titchmarsh function residues are declared to be the supplementary spectral data, see Definition~\ref{SpectralData} below.

In our approach a solution to the characterization problem for the mapping $\Lambda$ (in other words, a complete characterization of classes of spectral data corresponding to fixed classes of potentials and boundary conditions) consists of four disjoint parts:

\smallskip

\emph{(i) The mapping $\Lambda$ is one-to-one.} It is convenient to fix the projectors $T_\pm$ in (\ref{BoundConditionIntro}) but to include operators $a,b$ involved in these boundary conditions into the set of unknown parameters together with a potential $V$. Thus, in this part we claim that spectral data of operator $L$ determines $(V;a,b)$ uniquely. It is worth noting that one can reconstruct $T_\pm$ from rough asymptotics of the Weil-Titchmarsh function residues.

\smallskip

\emph{(ii) Algebraic independence and local changes of spectral data.} In this part we check that the spectral data are ``locally free parameters'', i.e., each particular one of them can be changed in an (almost) arbitrary way: the only one restriction is some ``algebraic nondegeneracy condition'', see~(\ref{NondegeneracyCond}). While in the scalar case those local changes are given by shifts of a single eigenvalue and changes of a particular Weyl-Titchmarsh function residue, the vector-valued problem admits richer behavior: eigenvalues having intermediate multiplicities can merge or split.

\smallskip

\emph{(iii) Sharp asymptotics of spectral data.} There is a principal difference between spectral asymptotics for the scalar problem and the vector-valued one. In particular, one needs to use some regularization of spectral data if no particular assumptions ensuring the asymptotic simplicity of eigenvalues are imposed, see further discussion below. Moreover, for boundary conditions (\ref{BoundConditionIntro}) rather involved asymptotics arise even in leading terms as geometry of mutual location of $H_-$ and $H_+$ comes into play.

\smallskip

\emph{(iv) The mapping $\Lambda$ is onto.} In this part we show that all data satisfying
algebraic nondegeneracy condition (\ref{NondegeneracyCond}) and asymptotics (iii) are spectral data of some potential from a given class.
Following \cite{PT,CKK,Che09} we prove that $\Lambda$ is onto in two steps. The first is a ``local surjection'' statement: some neighborhood of the unperturbed spectral data is contained in the image of $\Lambda$, which follows from (iii) and an abstract fixed point theorem. The second is an iteration of the procedure given in (ii) applied to a finite number of first eigenvalues, see further details below. 

\smallskip

For the sake of convenience, we present our results divided into four parts according to the list given above.

\smallskip

\noindent {\bf (i) Uniqueness Theorem.} In other words, we prove that the mapping $\Lambda$ is \mbox{$1$-to-$1$}.
This is the simplest and rather well investigated part, cf. \cite{Mal05, Yu06a, CKIsoSets}. In our
paper we adopt the method from \cite{PT} and prove that for an arbitrary pair of projectors $T_\pm$
spectral data of the Sturm-Liouville operator (\ref{SLPi})--(\ref{BoundConditionIntro}) determine $(V;a,b)$ uniquely.
Dealing with boundary conditions~(\ref{BoundConditionIntro}) we face with asymptotics of fundamental solutions that differ from the scalar case even in leading terms.
Roughly speaking, we are forced to work with an arbitrary mixture of Dirichlet, Neumann and general ``twisted'' (e.g., $\rank T_-=\rank T_+ =\rank T_-T_+=1$ for $N=2$) boundary conditions simultaneously, so an appropriate notation should be developed. In a sense, along with the uniqueness theorem proof given in this paper we set up a framework for the further sharp asymptotic analysis of spectral data given in \cite{ChMat-III}.

\smallskip

\noindent {\bf (ii) Local changes of spectral data, \cite{ChMat-II}}.
Recall that we choose the sequence of residues of the corresponding Weyl-Titchmarsh function to be the supplementary spectral data of the operator $L$. In \cite{ChMat-II} we show that these residues are {``locally free parameters''}, i.e., we can modify each one of them (and, further, any finite number) in an almost arbitrary way keeping $V(x)$ in a fixed class we are dealing with. The only one restriction for those local modifications is the algebraic nondegeneracy condition~(\ref{NondegeneracyCond}) given below. An algebraic restriction of this sort for spectral data of vector-valued Sturm-Liouville operators with Dirichlet boundary conditions appeared in \cite{CKIsoSets}, and (\ref{NondegeneracyCond}) provides a suitable version of that for general boundary conditions. It should be pointed out that two special cases of such local changes of spectral data are merging and splitting up of eigenvalues. This is a particular feature of the vector-valued problem: if $N=1$, all eigenvalues must be simple.

In the scalar setup explicit transforms of $(V;a,b)$ corresponding to the change of a single Weyl-Titchmarsh function residue or to the shift of a single eigenvalue can be constructed via the well known Darboux transform technique, see~\cite{IT,PT}. A similar technique was used in \cite{CKIsoSets} to describe the isospectral set for the vector-valued problem with Dirichlet boundary conditions. Unfortunately, it didn't allow to shift eigenvalues or to change their
 multiplicities which caused some technical problems in the characterization theorem proof given in \cite{CK} and prevented an
 immediate generalization of this characterization result to the case of an arbitrary mean potential.

 In \cite{ChMat-II} we use a more powerful tool (namely, the remarkable reconstruction procedure based on the so-called method of spectral mappings) developed by Yurko \cite{FYu01, Yu06a, Yu06b} and Bondarenko \cite{Bo11, Bo12}. Note that only the algebraic nature of this reconstruction
 algorithm matters and no hard analysis of spectral asymptotics is involved when one modifies a \emph{single} Weyl-Titchmarsh function residue.

\smallskip

\noindent {\bf (iii) Sharp asymptotics of spectral data, \cite{ChMat-III}}. It should be said that spectral data asymptotics of vector-valued Sturm-Liouville operators with summable potentials do \emph{not} mimic those for scalar operators. Even in the simplest case of Dirichlet boundary conditions asymptotics of Weyl-Titchmarsh function residues at series of eigenvalues $\pi^2n^2\!+\!v_j^0\!+\!o(1)$ corresponding to different eigenvalues $v_j^0$ of the mean potential $\int_0^1V(t)dt$ are not independent, see \cite[Eq.~(1.6)]{CK}. Namely, there are \emph{two} conditions that complement each other: rough asymptotics hold true for particular series corresponding to each of $v_j^0$ while some ``regularization'' (summation around~$\pi^2n^2$) should be used to specify next order terms. To simplify the analysis of such a dependence, in \cite{CK} it was assumed that all $v_j^0$ are pairwise distinct. Under this assumption \cite[Theorem~1.1]{CK} provides a complete description of the asymptotic structure of spectral data for $L^2$~potentials and Dirichlet boundary conditions.

In the subsequent paper~\cite{MT} Mikityuk and Trush gave the complete characterization of spectral data corresponding to matrix-valued distributional potentials in the Sobolev space~$W_2^{-1}$ using the Krein accelerant method. Note that the involved asymptotic structure of spectral data described above disappears for distributional potentials as there are no particular series of eigenvalues anymore and the summation around $\pi^2n^2$ is plugged in spectral asymptotics from the very beginning.

The next important contribution to the inverse spectral theory of vector-valued Sturm-Liouville operators was made by Bondarenko who applied a version of the method of spectral mappings to vector-valued operators with square summable potentials. In particular, \cite{Bo11,Bo12} contain the explicit reconstruction algorithms for Neumann-type and Dirichlet boundary conditions, respectively, without any restrictions on the mean potential. Unfortunately, there is a crucial inaccuracy in the asymptotics of spectral data given in these papers that is potentially dangerous for important estimates used along the analysis of the reconstruction procedure. Namely, \cite[Lemma~3,~Eq.~(4)]{Bo12} essentially claims that the behavior of single Weyl-Titchmarsh function residues mimic the scalar case which directly \emph{contradicts} to \cite[Theorem~1.1]{CK} if all $v_j^0$ are pairwise distinct: the correct necessary and sufficient asymptotics are \emph{weaker}. A similar mistake in the asymptotic structure of spectral data is contained in \cite{Bo11}: the last identities in the proof of \cite[Lemma~2]{Bo11} do \emph{not} imply the lemma statement.

Thus, even for Dirichlet boundary conditions some nontrivial analysis is needed to get the sharp asymptotic structure of spectral data. For general boundary conditions~(\ref{BoundConditionIntro}) such asymptotics are even more tricky to handle: now the series of eigenvalues are firstly grouped according to different $O(n)$ terms coming from the geometric interplay between projectors $T_-$ and $T_+$ (see Lemma~\ref{LemmaDecomposition} and Proposition~\ref{PropositionRoughAsymp} below), and only inside of these groups there is an influence of the mean potential which splits each group into particular series of eigenvalues of the form $\pi^2(n\!\pm\!\gamma_j)^2+\mathrm{const}+o(1)$. Moreover, as we do not impose any assumptions on the mean potential, there could be asymptotically merging series of eigenvalues and some further regularization is needed for such ``asymptotically multiple'' residues of the Weyl-Titchmarsh function. According to the overall strategy described above we postpone the careful analysis of spectral data asymptotics for operators (\ref{SLPi})--(\ref{BoundConditionIntro}) until the forthcoming paper~\cite{ChMat-III}.

\smallskip

\noindent {\bf (iv) Characterization theorem, \cite{ChMat-IV}}. This is the culminating part of our project. Here we show that sharp asymptotics obtained in (iii) together with the algebraic condition (\ref{NondegeneracyCond}) provide the complete characterization of spectral data corresponding to a fixed class of potentials, say, $V=V^*\in L^p([0,1];\C^{N\times N})$, $1\leq p<+\infty$. In general, there are two different strategies to prove results of this sort. The first one is to track some explicit reconstruction procedure carefully in order to show that, being started with data that satisfy necessary asymptotic and algebraic assumptions, it ends up with a potential from the class we are working with. Unfortunately, this way often leads to technical difficulties, especially if spectral asymptotics are complicated, e.g., to the best of our knowledge, no solution of this kind is known for the perturbed harmonic oscillator, cf.~\cite{CKK}. For the vector-valued Sturm-Liouville problem on a finite interval with $L^2$ potentials and Neumann-type or Dirichlet boundary conditions we believe that one can eventually save the reconstruction algorithm given in \cite{Bo11, Bo12} starting with correct (strictly weaker) assumptions on spectral data, though we cannot trace all technical details. Nevertheless, as even more involved technicalities appear for general boundary conditions (\ref{BoundConditionIntro}) and $L^p$ potentials, in our work we use the other approach which is briefly described below.

The second strategy to prove that the mapping $\Lambda$ is a bijection goes back to the work of Trubowitz and his co-authors, see \cite{PT}. As this mapping is already known from (i) to be $1$-to-$1$, it is sufficient to prove that $\Lambda$ is onto which is done in two steps independent of each other. Firstly, we show that all data satisfying (\ref{NondegeneracyCond}) and asymptotics obtained in (iii) which are \emph{close enough} to the spectral data of some reference potential (say, a~constant) can be obtained as spectral data of some potential from the class we are dealing with. In particular, this ``local surjection''  guarantees that for \emph{any} given data satisfying (\ref{NondegeneracyCond}) and (iii) there exists a proper potential whose spectral data coincide with the given ones for all sufficiently large energies. After that, the second step is purely algebraic: basing on the explicit step-by-step modification procedure (ii) we change a finite number of first Weyl-Titchmarsh function residues to be the given ones.

There are several ways to prove the ``local surjection'' statement. In \cite{PT}, the inverse function theorem was used for this purpose. Unfortunately, to apply this theorem one needs to prove in advance that the mapping $\Lambda$ is smooth (continuously differentiable in the Fr\'echet sense) everywhere in a vicinity of the reference potential which seems to be a quite hard technical problem in our case. Some way to overcome such technicalities was suggested in \cite{Che09}: namely, using the scalar Strum-Liouville problem as a model example, it was shown that one can derive the ``local surjection'' from the abstract Leray-Shauder-Tichonoff fixed point theorem and the differentiability of $\Lambda$ at a single point (reference potential).

In our forthcoming paper~\cite{ChMat-IV} we push this scheme even further and use only the fact that near a constant potential the mapping $\Lambda$  is close enough to some linear isomorphism. In particular, required estimates are not so sharp to establish the differentiability of $\Lambda$
at any point. Thus, instead of ``analytic'' properties of this mapping,
our approach is based on ``geometric'' analysis
in the domain and range spaces of~$\Lambda$. Though the present research is focused on the potential classes $L^p$,
we believe that our methods are universal and can be applied without major changes to other classes,
e.g., to Sobolev spaces $W_p^n$ (in this case a deeper hierarchical structure of eigenvalues series shows up and more involved analysis of spectral asymptotics should be done).

\smallskip

\noindent {\bf Organization of the paper.} The rest of this paper is devoted to the rough asymptotic analysis of eigenvalues of operators (\ref{SLPi})--(\ref{BoundConditionIntro}), especially their dependence on the mutual geometry of projectors $T_\pm$, and to the proof of the uniqueness theorem for these operators. We introduce the spectral data (residues of the proper Weyl-Titchmarsh function) and formulate our main results in Sect.~\ref{sec:mainresults}. Asymptotics of fundamental solutions, their Wronskian and its inverse are computed in Sect.~\ref{sec:W-1asymp}. The uniqueness theorem is proved in Sect.~\ref{sec:UniquenessPrf}, and Appendix contains several basic facts about matrix-valued Weyl-Titchmarsh functions corresponding to vector-valued Sturm-Liouville operators.

\smallskip

\noindent {\bf Acknowledgements.} This research was supported by the Chebyshev Laboratory at Saint Petersburg State University under the Russian Federation Government grant 11.G34.31.0026 and JSC~``Gazprom Neft''. The first author was partly supported by the RFBR grant 11-01-00584-a.

\section{Main results} \label{sec:mainresults}
\setcounter{equation}{0}
First of all, let us introduce some notation which is motivated by the theory of boundary triplets and proper extensions of closed symmetric operators in Hilbert spaces.
For orthogonal projectors $T_\pm$, $T_\pm^\perp=I-T_\pm$ and self-adjoint operators $a=T_-^\perp a T_-^\perp$, $b=T_+^\perp b T_+^\perp$ involved in the boundary conditions (\ref{BoundConditionIntro}) we set
\begin{equation} \label{Bcond}
\begin{aligned}
\Bcondm\psi&=T_-^\perp(\psi'(0)-a \psi(0))-T_-\psi(0),\quad&
\Bcondmd\psi&=T_-^\perp\psi(0)+T_-\psi'(0),\\
\Bcondp\psi&=T_+^\perp(\psi'(1)+b\psi(1))-T_+\psi(1),\quad&
\Bcondpd\psi&=T_+^\perp\psi(1)+T_+\psi'(1),&
\end{aligned}
\end{equation}
where a function $\psi=\psi(x)$, $x\in[0,1]$, can be either vector-valued or matrix-valued. In our project we consider the inverse spectral problem for  vector-valued Sturm-Liouville operators defined by the differential expression
\begin{equation}\label{SLP}
L\psi=-\psi''+V(x)\psi
\end{equation}
and boundary conditions
\begin{equation} \label{BoundConditions}
\Bcondm\psi=\Bcondp\psi=0,
\end{equation}
where $V=V^*\in L^1([0,1];\C^{N\times N})$ is a $N\times N$ matrix-valued potential. Due to the KLMN theorem (e.g., see ~\cite[\S{}X.2]{RS2}), one can define $L$ (via the corresponding quadratic form) as a self-adjoint operator acting in $L^2([0,1];\C^N)$.

It is well known that the spectrum of $L$ is purely discrete. Let
\[
\lambda_0<\lambda_1<\ldots<\lambda_\alpha<\ldots
\]
be the eigenvalues of $L$ (counted \emph{without} multiplicity) and $k_\alpha\in [1,N]$ denote the multiplicity of $\lambda_\alpha$. For a given potential $V$ and operators $a,b$ involved in the boundary conditions (\ref{BoundConditions}), let $\Fpm(x,\lambda)$ denote two $N\times N$ matrix-valued solutions of the differential equation
\begin{equation} \label{DiffEquation}
-\Psi''(x,\lambda)+V(x)\Psi(x,\lambda)=\lambda\Psi(x,\lambda)
\end{equation} such that
\begin{equation}\label{FpmBoundCond}
\begin{cases} \Bcondm\Fm=0,\\ \Bcondmd\Fm=I, \end{cases}\quad
\begin{cases} \Bcondp\Fp=0,\\ \Bcondpd\Fp=I. \end{cases}
\end{equation}
In other words,
$\Fm(0)=T_-^\perp,$ $\Fm'(0)=T_-+aT_-^\perp$ while $\Fp(1)=T_+^\perp,$ $\Fp'(1)=T_+-bT_+^\perp$.
One can easily check that $\Fpm(x,\lambda)$ are entire functions of $\lambda$ for each $x\in [0,1]$.


\begin{dfn}[\bf Spectral data]
 \label{SpectralData} For all $(V;a,b)$ and $\alpha\geq0$, let
 \[
\Pm_\alpha\colon\C^N\to\E_\alpha=\Ker (\Bcondp\Fm)(\lambda_\alpha)
 \]
denote the orthogonal projectors onto the subspaces $\E_\alpha$, and let the positive self-adjoint operators
$g_\alpha\colon\E_\alpha\to\E_\alpha$ be defined as
\begin{equation} \label{galpha}
\textstyle        g_\alpha = \left. G_\alpha\right|_{\E_\alpha}, \quad \text{where }\: G_\alpha=\Pm_\alpha  S_\alpha \Pm_\alpha\quad\text{and}\quad S_\alpha=\int\limits_0^1[\Fm(x,\lambda_\alpha)]^*\Fm(x,\lambda_\alpha) dx.
\end{equation}
We define the \emph{spectral data of $L$} to be the set of all triplets $\{(\lambda_\alpha,\,\Pm_\alpha,\,g_\alpha)\}_{\alpha\geq0}$.
\end{dfn}

\begin{rem} \label{ConnectionRanksMultiplicities}
Any vector-valued solution $\psi(x,\lambda)$ of the differential equation (\ref{DiffEquation}) satisfying the initial condition $\Bcondm \psi =0$ has the form $\Fm(x,\lambda)h$ for some $h\in\C^N$. Hence, all unnormalized eigenfunctions $\psi_\alpha(x)$ corresponding to the eigenvalue $\lambda_\alpha$ have the form
$\psi_\alpha=\Psi_\alpha h$, $\Psi_\alpha(x)=\Fm(x,\lambda_\alpha)\Pm_\alpha$, for some $h\in\E_\alpha$, and vice versa. In particular, it appears that
\[
\rank\Pm_\alpha=\dim \E_\alpha=k_\alpha.
\]
Moreover, the following identity is fulfilled:
\[
\textstyle \langle g_\alpha
h,h\rangle=\int\limits_0^1|\Psi_\alpha(x)h|^2dx, \quad h\in\E_\alpha,
\]
where $\langle u,v\rangle=u^*v$ is the scalar product of two vectors in $\C^N$ and $|u|^2=\langle u,u \rangle$.
\end{rem}

As usual (cf.~\cite{CKIsoSets}) the spectral data $\{(\lambda_\alpha,\,\Pm_\alpha,\,g_\alpha)\}_{\alpha\ge 0}$ can be thought about as a natural parameterization for residues of the \emph{matrix-valued Weyl-Titchmarsh function} corresponding to the differential operator $L$
\begin{equation}
\label{MFunctDef}
m(\lambda)=-(\Bcondmd\Fp)(\lambda){[(\Bcondm\Fp)(\lambda)]}^{-1}.
\end{equation}

\begin{prop} \label{resMfunction} For all $V=V^*\in L^1([0,1];\C^{N\times N})$ the Weyl-Titchmarsh function (\ref{MFunctDef}) is analytic in the complex plane $\C$ except simple poles at eigenvalues $\lambda_\alpha$, $\alpha\geq 0$, of the operator $L$. Moreover, $m(\ol{\lambda})=[m(\lambda)]^*$ for all $\lambda\in\C$, and
\[
\res\limits_{\lambda=\lambda_\alpha }m(\lambda)=-\Pm_\alpha g_\alpha^{-1}\Pm_\alpha\quad\text{for~all}~~\alpha\geq 0.
\]
\end{prop}

\begin{proof}
The proof is quite standard and mimics~\cite{CKIsoSets}. See Appendix for details.
\end{proof}

In our sequel paper \cite{ChMat-II} we show that Weyl-Tichmarsh function residues are ``locally free parameters'': if any particular one of them is changed in an almost arbitrary way, the new sequence of residues corresponds to some vector-valued Sturm-Liouville operator from the same class. Nevertheless, the next Proposition contains an algebraic restriction that prevents completely arbitrary changes of $\{(\lambda_\alpha,\,\Pm_\alpha)\}_{\alpha\ge 0}$.

\begin{prop} \label{DivisionByWronsk}
\noindent (i) Let $\zeta(\lambda)$ be an entire matrix-valued function and \mbox{$\zeta(\lambda_\alpha)\Pm_\alpha=0$} for all $\alpha\geq0$. Then
the function $\zeta(\lambda)[(\Bcondp\Fm)(\lambda)]^{-1}$ is entire.

\smallskip
\noindent (ii) \label{NondegeneracyCond}
For all $V=V^*\in L^1([0,1];\C^{N\times N})$ the spectral data $\{(\lambda_\alpha,\,\Pm_\alpha,\,g_\alpha)\}_{\alpha\geq0}$ satisfy the following
\emph{algebraic nondegeneracy condition}:
\begin{quotation}
\noindent there exists no entire vector-valued function $\xi\colon\C\to\C^N$ except $\xi(\lambda)\equiv 0$ such that $\Pm_\alpha\xi(\lambda_\alpha)=0$ for all $\alpha\geq0$,
$\xi(\lambda)=O(e^{|\im\sqrt\lambda|})$ as $|\lambda|\to\infty$, and
\begin{equation} \label{NondegeneracyCondAsymptotics}
T_-^\perp\xi(x)=O(x^{-\frac{1}{2}}),~~T_-\xi(x)=O(x^{-1})~~\text{as}~~x\to+\infty.
\end{equation}
\end{quotation}
\end{prop}

\begin{proof}
The proof of (i) essentially follows~\cite{CKIsoSets}, see Appendix for details. In order to derive~(ii) from~(i), note that, if such a vector-valued function~$\xi(\lambda)$ exists, then the function $\xi^\top(\lambda)[(\Bcondp\Fm)(\lambda)]^{-1}$ is entire. Then asymptotics of the inverse Wronskian as $\lambda\to\infty$ and the Liouville theorem imply $\xi(\lambda)\equiv 0$, see Sect.~\ref{subsec:W-1asymp} for details.
\end{proof}

The main result of this paper is the following uniqueness theorem for operators (\ref{SLP}) with separated boundary conditions (\ref{BoundConditions}) of general type. Note that in the existing literature the Dirichlet boundary conditions ($T_-^\perp=T_+^\perp=0$ ) and the Neumann-type ones ($T_-=T_+=0$) are mostly explored though some methods can be used for general separated boundary conditions too, see \cite[p.1140]{Yu06a}.

\begin{thm}[\bf Uniqueness] \label{uniqueness}
For each pair of orthogonal projectors $T_\pm$ spectral data of the
Sturm-Liouville operator (\ref{SLP})--(\ref{BoundConditions}) determine the parameters \mbox{$a$} and \mbox{$b$} which are involved in the boundary conditions (\ref{BoundConditions}) and the potential \mbox{$V=V^*\in L^1([0,1];\C^{N\times N})$} uniquely.
In other words, if spectral data $(\lambda_\alpha,\,\Pm_\alpha,\,g_\alpha)$ and $(\wt{\lambda}_\alpha,\,\wt{\Pm}_\alpha,\,\wt{g}_\alpha)$ corresponding to $(V;a,b)$ and $(\wt V;\wt a,\wt{b})$, respectively, coincide for all $\alpha\geq0$, then  $(V;a,b)=(\wt V;\wt a,\wt{b})$.
\end{thm}

\begin{proof}
See Sect.~\ref{sec:UniquenessPrf}.
\end{proof}

\section{Asymptotics of the inverse Wronskian as $\lambda\to\infty$\\
and the rough asymptotics of the spectrum}
\label{sec:W-1asymp}
\setcounter{equation}{0}

Let
\[
W(\lambda)=\{\Fp,\Fm\}(x,\lambda)=[\Fp(x,\ol{\lambda})]^*\Fm'(x,\lambda)-[\Fp'(x,\ol{\lambda})]^*\Fm(x,\lambda)
\]
denote the Wronskian of two fundamental solutions $\Fpm$ of the equation (\ref{DiffEquation}) satisfying the initial conditions (\ref{FpmBoundCond}). It is well known that the right hand side does not depend on $x$ (see Appendix), therefore,
\begin{align*}
W(\lambda)&=T_+^\perp\Fm'(1,\lambda)-(T_+\!-\!bT_+^\perp)\Fm(1,\lambda)=(\Bcondp\Fm)(\lambda) \cr
& = [\Fp(0,\ol{\lambda})]^*(T_-\!+\!aT_-^\perp)- [\Fp'(x,\ol{\lambda})]^*T_-^\perp= -[(\Bcondm\Fp)(\ol\lambda)]^*.
\end{align*}
The main goal of this Section is to compute asymptotics of the inverse Wronskian $[W(\lambda)]^{-1}$ as $\lambda\to\infty$. Since the residues of the Weyl-Titchmarsh function (\ref{MFunctDef}) are located exactly at zeros of $\det W(\lambda)$, along the way we also get rough asymptotics of the eigenvalues $\lambda_\alpha$ (see Proposition~\ref{PropositionRoughAsymp} below) and prove Proposition~\ref{DivisionByWronsk}(ii). Clearly, the Wronskian asymptotics should highly depend on the mutual geometry of projectors $T_\pm$ involved in the boundary conditions (\ref{BoundConditions}). Thus, we start with choosing a special ``coordinate system'' depending on $T_\pm$ in order to simplify the further computations.

\subsection{Coordinate system} For the sake of convenience from now onwards, if the letter $\Es$ with some indices stands for a subspace of $\C^N$, then the letter $\P$ with the same indices denotes the orthogonal projector $P:\C^N\to E$ onto this subspace.

\begin{dfn}  Let subspaces~$\Es_1,\Es_2\subset\C^N$ be the ranges of orthogonal projectors $\P_1,\P_2$ acting in $\C^N$, and $\dim\Es_1=\dim\Es_2$. We write 
$\angle(\Es_1,\Es_2)=\gamma\in(0,\tfrac{\pi}{2})$
if $\|P_2x_1\|=\cos\gamma\cdot \|x_1\|$ for all $x_1\in\Es_1$ and $\|P_1x_2\|=\cos\gamma\cdot\|x_2\|$ for all $x_2\in\Es_2$.
\end{dfn}

\begin{lem}
\label{LemmaDecomposition}
There exists a unique orthogonal decomposition 
\begin{equation} \label{OrthogonalDecomposition}
\C^N=\Es^{\EuScript{DD}}\oplus \Es^{\EuScript{ND}}\oplus \Es^{\EuScript{DN}}\oplus \Es^{\EuScript{NN}}\oplus\bigoplus\limits_{i=1}^k \Es[i],
\end{equation}
where
\[
\begin{aligned}
E^{\EuScript{DD}}&=\Ran T_-\cap \Ran T_+,& E^{\EuScript{ND}}&=\Ran T^\perp_-\cap \Ran T_+,\\
E^{\EuScript{DN}}&=\Ran T_-\cap \Ran T_+^\perp,& E^{\EuScript{NN}}&=\Ran T^\perp_-\cap \Ran T_+^\perp,
\end{aligned}
\]
such that $0<\gamma_1<\gamma_2<\ldots<\gamma_k<\frac\pi2$ and for all $i=1,\dots,k$ the following is fulfilled:
\[
\Es[i]=\Es[i]-+\Es[i]+, \quad \angle(\Es[i]-,\Es[i]+)=\gamma_i,\quad \text{where}\quad \Es[i]_\pm =\Es[i]\cap\Ran T_\pm.
\]
In particular, $\dim\Es[i]_-=\dim\Es[i]_+$ for all $i=1,\dots,k$.
\end{lem}

\begin{rem} \label{RemarkEmEp}
Below we need the notation
\begin{equation*} \label{EmAndEp}
\Es-=\bigoplus\limits_{i=1}^k \Es[i]-,\qquad \Es+=\bigoplus\limits_{i=1}^k \Es[i]+,
\qquad \Es^{\EuScript{T}}=
\bigoplus\limits_{i=1}^k \Es[i].
\end{equation*}
and the corresponding notation for the orthogonal projectors $\PEpm$, $P^{\EuScript{T}}$ onto $\Es_\pm$, $\Es^{\EuScript{T}}$. We also
use the orthogonal projectors $\PEpm^\perp=P^{\EuScript{T}}-P_\pm$ onto the subspaces $\Es_\pm^\perp=\Es^{\EuScript{T}}\ominus\Es_\pm$. It is easy to see that
\begin{equation}
\label{rankPmPp}
\rank\PEp=\rank\PEm=\rank\PEp\PEm=\rank\PEp\PEm^\perp=\rank\PEp^\perp\PEm=\rank\PEp^\perp\PEm^\perp.
\end{equation}
\end{rem}

\begin{proof}
The Hilbert-Schmidt theorem applied to the self-adjoint operator $\Tp\Tm\Tp$ gives
\begin{equation} \label{TpTmTpSpectralDecompose}
\Tp\Tm\Tp x=
\sum\limits_{i=0}^k\sum\limits_{j=1}^{l_i} \cos^2\gamma_i\cdot\langle x, \hb[i]_j\rangle \hb[i]_j,\quad x\in\C^N,
\end{equation}
where $\lambda_i=\cos^2\gamma_i$ denote the nonzero eigenvalues of $\Tp\Tm\Tp$ (for the sake of convenience, we formally include $\gamma_0=0$ into this sum even if $1$ is not an eigenvalue), and $\hb[i]_j$ are the orthonormal eigenvectors corresponding to $\lambda_i$ (we set $l_0=0$ if $\|\Tp\Tm\Tp\|<1$).

As $\Tp\Tm\Tp=(\Tm\Tp)^*(\Tm\Tp)$, it's easy to see that the vectors
\begin{equation} \label{EigenvectorTmTp}
\eb[i]_j={(\cos\gamma_i)^{-1}}\cdot {\Tm\Tp \hb[i]_j}
\end{equation}
form an orthonormal basis in $\Ran \Tm\Tp = (\Ker \Tp\Tm)^\perp$ and 
\[
\Tp\Tm x=\sum\limits_{i=0}^k\sum\limits_{j=1}^{l_i} \cos\gamma_i\cdot \langle x, \eb[i]_j\rangle \hb[i]_j,\quad x\in\C^N.
\]
Therefore,
\begin{equation*}\begin{split}
&\Tm\Tp x=(\Tp\Tm)^*x=\sum\limits_{i=0}^k\sum\limits_{j=1}^{l_i}\cos\gamma_i\cdot\langle x, \hb[i]_j\rangle \eb[i]_j,\quad x\in\C^N, \\
&\Tm\Tp\Tm x=\sum\limits_{i=0}^k\sum\limits_{j=1}^{l_i}\cos^2\gamma_i\cdot\langle x, \eb[i]_j\rangle \eb[i]_j,\quad x\in\C^N.
\end{split}\end{equation*}
Hence, the operators $\Tm\Tp\Tm$ and $\Tp\Tm\Tp$ have the same spectra and
$\eb[i]_j$ are the orthonormal eigenvectors of $\Tm\Tp\Tm$ corresponding to the
eigenvalue $\cos^2\gamma_i$. 

Given $i=0,\dots,k$, let $\Es[i]-$ and $\Es[i]+$ denote the eigenspaces of $\Tm\Tp\Tm$ and $\Tp\Tm\Tp$,
respectively, corresponding to the eigenvalue $\cos^2\gamma_i$. It is clear that \mbox{$\Es[0]_-\!=\Es[0]_+\!=\Es^{\EuScript{DD}}$},  \mbox{$\Es[i]_\mp\subset\Ran T_\mp$} and \mbox{$\Es[i]_\mp\cap\Ran T_\pm=\{0\}$} for $i=1,\dots,k$. The identities (\ref{TpTmTpSpectralDecompose}) and (\ref{EigenvectorTmTp}) imply
\[
\cos^2\gamma_i\cdot \delta_{ij}\delta_{sl}=\langle \Tp\Tm\Tp \hb[i]_s,\hb[j]_l\rangle=
\langle \Tm\Tp \hb[i]_s,\hb[j]_l\rangle=\cos\gamma_i\cdot \langle \eb[i]_s, \hb[j]_l\rangle,
\]
where $\delta_{ij}$ and $\delta_{sl}$ are the Kronecker deltas. Hence, $\angle(\Es[i]-,\Es[i]+)=\gamma_i$ and
\[
\Ran \Tp\Tm\Tp + \Ran \Tm\Tp\Tm=\Es^{\EuScript{DD}}\oplus\bigoplus\limits_{i=1}^k \Es[i], \quad \Es[i]=\Es[i]- +\Es[i]+. 
\]
To finish the proof of (\ref{OrthogonalDecomposition}) it is sufficient to check that
\[
(\Ran \Tp\Tm\Tp + \Ran \Tm\Tp\Tm)^\perp=
\Ker \Tm\Tp\cap\Ker \Tp\Tm = E^{\EuScript{ND}}\oplus E^{\EuScript{DN}}\oplus E^{\EuScript{NN}}.
\]
By definition, $E^{\EuScript{ND}}\oplus E^{\EuScript{DN}}\oplus E^{\EuScript{NN}}\subset \Ker \Tm\Tp\cap\Ker \Tp\Tm$. At the same time, if  $\Tm\Tp h=\Tp\Tm h=0$, then one has the decomposition
\[
h=\Tp h + \Tm h + (h\!-\!\Tp h\!-\Tm h) = h^{\EuScript{ND}}+ h^{\EuScript{DN}}+ h^{\EuScript{NN}}
\]
where $h^{\EuScript{ND}}\in \Ker\Tm\cap\Ran\Tp=E^{\EuScript{ND}}$, $h^{\EuScript{DN}}\in E^{\EuScript{DN}}$, and \mbox{$h^{\EuScript{NN}}\in\Ker\Tm\cap\Ker\Tp=E^{\EuScript{NN}}$}.

In order to prove the uniqueness of the decomposition (\ref{OrthogonalDecomposition}) note that the identities
$\Tm=\P^{\EuScript{DD}}+\P^{\EuScript{DN}}+\sum_{i=0}^k \P[i]-$,
$\Tp=\P^{\EuScript{DD}}+\P^{\EuScript{ND}}+\sum_{i=0}^k \P[i]+$ imply
the representations
\[
\Tm\Tp\Tm=\P^{\EuScript{DD}}+\sum\limits_{i=0}^k \cos^2\gamma_i\cdot\P[i]-,\qquad
\Tp\Tm\Tp=\P^{\EuScript{DD}}+\sum\limits_{i=1}^k\cos^2\gamma_i\cdot\P[i]+,
\]
which are unique due to the spectral theorem.
\end{proof}

\subsection{Asymptotics of the fundamental solutions and $W(\lambda)$}
 Repeating the standard arguments (see \cite{PT}, \mbox{p.~13-15}) one obtains the following (uniform on $x\in [0,1]$ and bounded subsets of $V\in L^1([0,1];\C^{N\times N})$) asymptotics as $|\lambda|\to\infty$:

\begin{equation}
\label{asymptoticsFminus} {\small
\begin{split}
&\Fm(x,\lambda)=\left[\frac{\sin(\sqrt\lambda x)}{\sqrt\lambda} I- \frac{\cos(\sqrt\lambda
x )}{2\lambda}\int\limits_0^x V(t)dt +
o\left(\frac{e^{|\im\sqrt\lambda|x}}{\lambda}\right)\right]T_- \\
&\quad\quad\quad\quad\quad +\!\left[\cos(\sqrt{\lambda}x)I+\frac{\sin(\sqrt{\lambda}x)}{\sqrt\lambda}
\left(a+\frac 12 \int\limits_0^x V(t)dt\right)+
o\left(\frac{e^{|\im\sqrt\lambda|x}}{\sqrt\lambda}\right) \right]T_-^\perp,\\
&\Fm'(x,\lambda)=\left[\cos(\sqrt\lambda x) I+ \frac{\sin(\sqrt\lambda
x)}{2\sqrt\lambda}\int\limits_0^x V(t) dt
+o\left(\frac{e^{|\im\sqrt\lambda|x}}{\sqrt\lambda}\right)\right]T_-\\
&\quad\quad\quad\quad\quad +\!\left[-\sqrt\lambda\sin(\sqrt\lambda x)I+\cos(\sqrt\lambda x) \left(
a+\frac 12\int\limits_0^x V(t) dt\right)+
o\left({e^{|\im\sqrt\lambda|x}}\right)\right]T_-^\perp,
\end{split}}
\end{equation}
\begin{equation}
\label{asymptoticsFplus} {\small
\begin{split}
&\Fp (x,\lambda)=\left[-\frac{\sin(\sqrt\lambda (1\!-\!x))}{\sqrt\lambda} I+
\frac{\cos(\sqrt\lambda(1\!-\!x))}{2\lambda}\int\limits_x^1 V(t)dt +
o\left(\frac{e^{|\im\sqrt\lambda|(1-x)}}{\lambda}\right)\right]T_+ \\
&+\!\left[\cos(\sqrt{\lambda}(1\!-\!x))I+\frac{\sin(\sqrt{\lambda}(1\!-\!x))}{\sqrt\lambda}
\left(\frac 12 \int\limits_x^1 V(t)dt+b\right)+
o\left(\frac{e^{|\im\sqrt\lambda|(1-x)}}{\sqrt\lambda}\right) \right]T_+^\perp,\\
&\Fp'(x,\lambda)=\left[\cos(\sqrt\lambda (1\!-\!x)) I+ \frac{\sin(\sqrt\lambda
(1\!-\!x))}{2\sqrt\lambda}\int\limits_x^1 V(t) dt
+o\left(\frac{e^{|\im\sqrt\lambda|(1-x)}}{\sqrt\lambda}\right)\right]T_+  \\
&+\!\left[\sqrt\lambda\sin(\sqrt\lambda (1\!-\!x))I-\cos(\sqrt\lambda (1\!-\!x))
\left( \frac 12\int\limits_x^1 V(t) dt+b\right)+
o\left({e^{|\im\sqrt\lambda|(1-x)}}\right)\right]T_+^\perp.\\
\end{split} }\end{equation}
In particular,
\[
\begin{aligned}
W(\lambda)=(-\sqrt\lambda\sin\sqrt\lambda+O(e^{|\im\sqrt{\lambda}|}))\cdot\Tp^\perp\Tm^\perp &+ (\cos\sqrt\lambda+O(|\lambda|^{-\frac{1}{2}}e^{|\im\sqrt{\lambda}|}))\cdot\Tp^\perp\Tm \\
+(-\cos\sqrt\lambda+O(|\lambda|^{-\frac{1}{2}}e^{|\im\sqrt{\lambda}|}))\cdot\Tp\Tm^\perp & +(-\tfrac{\sin\sqrt\lambda}{\sqrt\lambda}+O(|\lambda|^{-1}e^{|\im\sqrt{\lambda}|}))\cdot \Tp\Tm\,.
\end{aligned}
\]
It is more convenient to use the matrix notation, so we rewrite $W(\lambda)$ as
\begin{equation} \label{WAsympthotic}
W(\lambda)=W_0(\lambda)+\qForm{\Tp^\perp}{\Tp}{O(1)}{O(\lambda^{-\frac{1}{2}})}
{O(\lambda^{-\frac{1}{2}})}{O(\lambda^{-1})}{\Tm^\perp}{\Tm}\cdot e^{|\im\sqrt{\lambda}|},
\end{equation}
where
\[
W_0(\lambda)=-\sqrt\lambda\sin\sqrt\lambda\cdot\Tp^\perp\Tm^\perp +\cos\sqrt\lambda\cdot\Tp^\perp\Tm-\cos\sqrt\lambda\cdot\Tp\Tm^\perp-\tfrac{\sin\sqrt\lambda}{\sqrt\lambda}\cdot \Tp\Tm\,.
\]
One should be careful when writing down the leading terms: in (\ref{WAsympthotic}), blocks in the error term are not quadratic if no special assumptions are imposed on the ranks of $T_\pm$. In order to handle this situation properly, we introduce the following notation: let
 \begin{equation} \label{Id}
 I=\qForm{\PEp^\perp}{\PEp}{I_{11}}{I_{12}}{I_{21}}{I_{22}}{\PEm^\perp}{\PEm}\colon
 \Es^{\EuScript{T}}\to\Es^{\EuScript{T}}
\end{equation}
be the identical map (recall that $\Es^{\EuScript{T}}=\Em+\Ep$ but this sum is \emph{not} orthogonal, see Lemma~\ref{LemmaDecomposition} and Remark~\ref{RemarkEmEp}), where
\[
\begin{aligned}
 I_{11}&=\PEp^\perp \PEm^\perp\colon \Em^\perp\to \Ep^\perp,&\quad
 I_{12}&=\PEp^\perp \PEm\colon \Em\to \Ep^\perp, \\
 I_{21}&=\PEp \PEm^\perp\colon \Em^\perp\to \Ep,&\quad
 I_{22}&=\PEp \PEm\colon \Em\to \Ep.
\end{aligned}
\]
It follows from (\ref{rankPmPp}) that {all operators $I_{11},I_{12},I_{21},I_{22}$ are invertible}. Recall that
\[
\begin{pmatrix} T_+^\perp \\ T_+ \end{pmatrix} = \begin{pmatrix} P^{\EuScript{DN}}+P^{\EuScript{NN}}+P_+^\perp \\ P^{\EuScript{DD}}+P^{\EuScript{ND}}+P_+\end{pmatrix},\qquad \begin{pmatrix} T_-^\perp \\ T_- \end{pmatrix} = \begin{pmatrix} P^{\EuScript{ND}}+P^{\EuScript{NN}}+P_-^\perp \\ P^{\EuScript{DD}}+P^{\EuScript{DN}}+P_+\end{pmatrix}.
\]
Using this notation the unperturbed Wronskian can be rewritten as
\begin{equation} \begin{split}\label{Wnol}
W_0(\lambda)= -\sqrt\lambda\sin\sqrt\lambda\cdot P^{\EuScript{NN}}-\cos\sqrt\lambda \cdot P^{\EuScript{ND}}
+\cos\sqrt\lambda\cdot P^{\EuScript{DN}}-\frac{\sin\sqrt\lambda}{\sqrt\lambda} \cdot P^{\EuScript{DD}} & \\+
\qForm{\PEp^\perp}{\PEp}
{-\sqrt\lambda\sin\sqrt\lambda\cdot I_{11}}{\cos\sqrt\lambda\cdot I_{12}}
{-\cos\sqrt\lambda \cdot I_{21}}{-\frac{\sin\sqrt\lambda}{\sqrt\lambda} \cdot I_{22} }
{\PEm^\perp}{\PEm}&.
\end{split} \end{equation}

\subsection{Inverse Wronskian and the rough asymptotics of the spectrum}
\label{subsec:W-1asymp}

\begin{lem} \label{WnolInverseLemma}
The inverse matrix to $W_0(\lambda)$ has the representation
\begin{equation} \label{WnolInverse}
\begin{split}
[W_0(\lambda)]^{-1}= -\tfrac{1}{\sqrt\lambda\sin\sqrt\lambda} \cdot P^{\EuScript{NN}}-\tfrac{1}{\cos\sqrt\lambda} \cdot P^{\EuScript{ND}}
+\tfrac{1}{\cos\sqrt\lambda}\cdot P^{\EuScript{DN}}-\tfrac{\sqrt\lambda}{\sin\sqrt\lambda} \cdot P^{\EuScript{DD}} &\\
+ \qForm{\PEm^\perp}{\PEm}
{-\tfrac{1}{\sqrt\lambda\sin\sqrt\lambda}\cdot J_{11}(\lambda)}
{-\tfrac{\cos\sqrt\lambda}{\sin^2\sqrt\lambda}\cdot J_{12}(\lambda)}
{\tfrac{\cos\sqrt\lambda}{\sin^2\sqrt\lambda}\cdot J_{21}(\lambda)}
{-\tfrac{\sqrt\lambda}{\sin{\sqrt\lambda}}\cdot J_{22}(\lambda)}
{\PEp^\perp}{\PEp} &,
\end{split}
\end{equation}
where the matrix-valued functions $J_{ij}(\lambda)$ acting from $\Ep^\perp$ or $\Ep$ to $\Em^\perp$ or $\Em$, respectively, are analytic for all $\lambda\in\C$ except the set $\sigma^\EuScript{T}=\{(\pi n\pm\gamma_i)^2,n\geq 0, i=1,\dots,k\}$. Moreover, the following uniform estimates are fulfilled for all $\lambda\in\mathbb{C}\setminus \sigma^\EuScript{T}$:
\begin{equation}
\label{JijEstim}
\|J_{11}(\lambda)\|,\|J_{12}(\lambda)\|,\|J_{21}(\lambda)\|,\|J_{22}(\lambda)\|= O \biggl(1+\frac{|\lambda|^{\frac{1}{2}}}{\dist(\lambda;\sigma^\EuScript{T})}\biggr).
\end{equation}
In particular, for each fixed constant $\beta\ne\pm\gamma_1,\dots,\pm\gamma_k\mod \pi$,
the norms $\|J_{ij}(\lambda)\|$ are uniformly bounded on the contours
$C_n^{(\beta)}=\{\lambda:|\lambda|=(\pi n+\beta)^2\}$ as $n\to\infty$.
\end{lem}
\begin{proof}
Applying the Frobenius formula (e.g., see \cite[Chapter~2.5]{Gant})
\[ 
\begin{pmatrix} m_{11} & m_{12}\\m_{21}&m_{22} \end{pmatrix}^{-1}=
\begin{pmatrix} (m_{11}\!-\!m_{12}m_{22}^{-1}m_{21})^{-1}&-(m_{11}\!-\!m_{12}m_{22}^{-1}m_{21})^{-1}m_{12}m_{22}^{-1}\\
-(m_{22}\!-\!m_{21}m_{11}^{-1}m_{12})^{-1}m_{21}m_{11}^{-1}&(m_{22}\!-\!m_{21}m_{11}^{-1}m_{12})^{-1}\end{pmatrix}
\] 
to (\ref{Wnol}) one obtains (\ref{WnolInverse}) with
\[
\begin{array}{ll}
J_{11}(\lambda)=[I_{11}+\cot^2\sqrt\lambda\cdot I_{12}I_{22}^{-1}I_{21}]^{-1}, & J_{12}(\lambda)=-J_{11}(\lambda)I_{12}I_{22}^{-1},\\
J_{21}(\lambda)=-J_{22}(\lambda)I_{21}I_{11}^{-1}, & J_{22}(\lambda)=[I_{22}+\cot^2\sqrt\lambda\cdot I_{21}I_{11}^{-1}I_{12}]^{-1}.
\end{array}
\]
Since the operators $I_{11}^{-1}$, $I_{22}^{-1}$, $I_{12}$ and $I_{21}$ are bounded, it is sufficient to estimate $\|J_{11}(\lambda)\|$ and $\|J_{22}(\lambda)\|$. We prove the inequality (\ref{JijEstim}) for the latter, the proof for the former is similar. Using the orthogonal basis $\hb[i]_j$ of the subspace $\Ep$ introduced in the proof of Lemma~\ref{LemmaDecomposition} it is easy to see that
\[
J_{22}(\lambda)x=\sum_{i=1}^k\sum_{j=1}^{l_i} [\cos\gamma_i -\cot^2\sqrt{\lambda}\cdot \tfrac{\sin^2\gamma_i}{\cos\gamma_i}]^{-1}\cdot \langle x, \hb[i]_j\rangle \hb[i]_j,\quad x\in\Ep.
\]
Hence,
\[
\|J_{22}(\lambda)\|=\max_{i=1,\dots,k}\left|\cos\gamma_i -\cot^2\sqrt{\lambda}\cdot \tfrac{\sin^2\gamma_i}{\cos\gamma_i}\right|^{-1}= \max_{i=1,\dots,k}\frac{\cos\gamma_i\cdot|\sin\sqrt\lambda|^2}{|\sin(\sqrt\lambda\!-\!\gamma_i)\sin(\sqrt\lambda\!+\!\gamma_i)|}.
\]
It is easy to see that there exist two constants $c_{1,2}>0$ such that
\begin{equation}
\label{SinEstim}
c_1\leq \frac{|\sin z|}{e^{|\im z|}\cdot\min \{1,\dist(z;\pi \Z)\}}\leq c_2\quad \text{for~all}~z\in\C\,.
\end{equation}
As two distances $\dist(\sqrt\lambda\!\pm\!\gamma_i;\pi\Z)$ cannot be small simultaneously, this implies
\[
\frac{|\sin\sqrt\lambda|^2}{|\sin(\sqrt\lambda\!-\!\gamma_i)\sin(\sqrt\lambda\!+\!\gamma_i)|}
=O(1+[\dist(\sqrt\lambda\!+\!\gamma_i;\pi \Z)]^{-1}+[\dist(\sqrt\lambda\!-\!\gamma_i;\pi Z)]^{-1}).
\]
In order to finish the proof it is sufficient to note that, for all $\lambda\ne (\pi n\pm\gamma_i)^2$, one has
\[
[\dist(\sqrt\lambda\!+\!\gamma_i;\pi \Z)]^{-1}+[\dist(\sqrt\lambda\!-\!\gamma_i;\pi \Z)]^{-1}= O\biggl(\frac{1+|\lambda|^\frac{1}{2}}{\dist(\lambda;\{(\pi n\!\pm\!\gamma_i)^2,n\in\Z\})}\biggr)\,.\qedhere
\]
\end{proof}

\begin{proof}[Proof of Proposition~\ref{NondegeneracyCond}(ii)]
Let an entire function $\xi\colon\C\to\C^N$ satisfy \mbox{$\Pm_\alpha\xi(\lambda_\alpha)=0$} for all \emph{$\alpha\geq0$}. Then, it follows from (i) that the function $\xi^\top(\lambda)[W(\lambda)]^{-1}$ is entire.

Let $\beta\in (0,\frac{\pi}{2})\setminus\{\gamma_1,\dots,\gamma_k\}$ be a fixed constant.
For $\lambda\in C_n^{(\beta)}$, $n\to\infty$, Lemma~\ref{WnolInverseLemma} and asymptotics (\ref{WAsympthotic}) give
\begin{equation}  \label{InverseW}
\begin{split}
[W(\lambda)]^{-1}&=[W_0(\lambda)]^{-1}+[W_0(\lambda)]^{-1}(W(\lambda)-W_0(\lambda))[W_0(\lambda)]^{-1}\\ &=[W_0(\lambda)]^{-1}+\qForm{\Tm^\perp}{\Tm}{O(\lambda^{-1})}{O(\lambda^{-\frac{1}{2}})}
{O(\lambda^{-\frac{1}{2}})}{O(1)}{\Tp^\perp}{\Tp}\cdot e^{-|\im\sqrt{\lambda}|}.
\end{split}
\end{equation}
It directly follows from our assumptions on $\xi(\lambda)$ and asymptotics~(\ref{WnolInverse}),~(\ref{InverseW}) that $\xi^\top(\lambda)[W(\lambda)]^{-1}=O(\lambda^{{1}/{2}})$ for $\lambda\in C_n^{(\beta)}$ as $n\to\infty$ and $\xi^\top(x)[W(x)]^{-1}=O(x^{-{1}/{2}})$ as $x\to+\infty$. Therefore, $\xi(\lambda)\equiv 0$ due to the Liouville theorem.
\end{proof}

\begin{prop}[{\bf rough asymptotics of the spectrum}] \label{PropositionRoughAsymp}
For sufficiently large energies, the spectrum $\{\lambda_\alpha\}_{\alpha\ge 0}$ splits into series according to the orthogonal decomposition of the space $\C^N$ described in Lemma~\ref{LemmaDecomposition}. Namely, there exists an integer $n^\bullet=n^\bullet(\|V\|)\geq 1$ and a constant $\varpi=\varpi(\|V\|)\geq 0$ such that, for all $n\geq n^\bullet$, the operator (\ref{SLP})--(\ref{BoundConditions}) has exactly
\begin{itemize}
\item $nN +N^{\EuScript{NN}}$ eigenvalues on the interval $(-\infty,\pi^2n^2+\varpi)$,
\item $N^{\EuScript{NN}}+ N^{\EuScript{DD}}$ eigenvalues on the interval $(\pi^2n^2-\varpi,\pi^2n^2+\varpi)$,
\item $N^{\EuScript{ND}}+N^{\EuScript{DN}}$ eigenvalues on the interval $(\pi^2(n+\tfrac{1}{2})^2-\varpi,\pi^2(n+\tfrac{1}{2})^2+\varpi)$,
\item $N^{[\gamma_i]}$ eigenvalues on each of the intervals $(\pi^2(n\!\pm\!\gamma_i)^2-\varpi,\pi^2(n\!\pm\!\gamma_i)^2+\varpi)$,
\end{itemize}
where $i=1,\dots,k$, $N^{[\gamma_i]}=\dim \Es[i]-=\dim \Es[i]+$, $N^{\EuScript{NN}}=\dim \Es^{\EuScript{NN}}$, $N^{\EuScript{DD}}=\dim \Es^{\EuScript{DD}}$, $N^{\EuScript{ND}}=\dim \Es^{\EuScript{ND}}$, $N^{\EuScript{DN}}=\dim \Es^{\EuScript{DN}}$, and
all eigenvalues are counted \emph{with} multiplicities.
\end{prop}

The proof is based on the following matrix version of Rouche's theorem:
\begin{lem} \label{RoucheMatrixLemm}
Let $F,G\colon \ol{B}\to\C^{N\times N}$ be analytic matrix-valued functions defined on some closed disc $\ol B$
and $\|G(\lambda)F^{-1}(\lambda)\|<1$ for all $\lambda\in\partial B$. Then, the scalar  functions $\det F$ and $\det(F+G)$ have the same number of zeroes in $B$ counting with multiplicities.
\end{lem}
\begin{proof}
For instance, see the proof of \cite[Lemma~2.2]{CK}.
\end{proof}

\begin{proof}[Proof of Proposition~\ref{PropositionRoughAsymp}] We take $F=W_0$ and $G=W-W_0$. The asymptotics (\ref{WAsympthotic}) and (\ref{WnolInverse}) together with (\ref{JijEstim}) and the estimate (\ref{SinEstim}) applied to the functions $\sin\sqrt\lambda$ and $\cos\sqrt\lambda=
\pm \sin(\sqrt\lambda\pm\frac{\pi}{2})$ imply that
\[
(W(\lambda)-W_0(\lambda))[W_0(\lambda)]^{-1} = O\left(|\lambda|^{-\frac{1}{2}}+\frac{1}{\dist(\lambda;\sigma^{\EuScript{T}}\cup\{\pi^2m^2,m\in\tfrac{1}{2}\Z\})}\right)
\]
as $\lambda\to\infty$, uniformly on bounded subsets of $L^1([0,1];\C^{N\times N})$. Thus, if $n^\bullet$ and $\varpi$ are chosen large enough, the estimate \mbox{$\|(W(\lambda)-W_0(\lambda))[W_0(\lambda)]^{-1}\|<1$} holds true on all the contours $\{\lambda:|\lambda|=\pi^2n^2+\varpi\}$ and $\{\lambda:|\lambda-\pi^2(n\!\pm\!\gamma_*)^2|=\varpi\}$, where $n\geq n^\bullet$ and $\gamma_*=0,\gamma_1,\dots,\gamma_k,\tfrac{1}{2}\pi$.
It is easy to see that
\[
\begin{split}
\det W_0(\lambda)&=[-\sqrt\lambda\sin\sqrt\lambda\,]^{\dim\Es^{\EuScript{NN}}}\cdot
[-\cos\sqrt\lambda\,]^{\dim\Es^{\EuScript{ND}}}\cdot[\,\cos\sqrt\lambda\,]^{\dim\Es^{\EuScript{DN}}}\\
&\times \biggl[-\frac{\sin\sqrt\lambda}{\sqrt\lambda}\,\biggr]^{\dim\Es^{\EuScript{DD}}}\cdot \prod\limits_{i=1}^k\:[\,\sin(\sqrt\lambda\!-\!\gamma_i)\sin(\sqrt\lambda\!+\!\gamma_i)]^{\dim\Es[i]_\pm}.
\end{split}
\]
Indeed, Lemma~\ref{LemmaDecomposition} ensures that $\det W_0(\lambda)$ factorizes into the product of terms corresponding to subspaces $\Es^{\EuScript{NN}}$, $\Es^{\EuScript{ND}}$, $\Es^{\EuScript{DN}}$, $\Es^{\EuScript{DD}}$ and $\Es[i]$, $i=1,\dots,k$. Thus, the first four factors appear immediately and to get the last product it is sufficient to note that
\[
\begin{split}
\det  \left[\qForm{P^{[\gamma_i]}-P^{[\gamma_i]}_+}{P^{[\gamma_i]}_+}
{-\sqrt\lambda\sin\sqrt\lambda\cdot I_{11}}{\cos\sqrt\lambda\cdot I_{12}}
{-\cos\sqrt\lambda \cdot I_{21}}{-\frac{\sin\sqrt\lambda}{\sqrt\lambda} \cdot I_{22} }
{P^{[\gamma_i]}-P^{[\gamma_i]}_-}{P^{[\gamma_i]}_-}\right] & \\ =[\sin^2\sqrt\lambda\cdot\cos^2\gamma_i +\cos^2\sqrt\lambda\cdot (-\sin^2\gamma_i)]^{\dim \Es[i]_\pm} &
\end{split}
\]
since $\angle(\Es[i]-,\Es[i]+)=\gamma_i$\,. In particular, zeroes of the entire function $\det W_0(\lambda)$ coincide with the points $\pi n$, $\pi n+\gamma_i$, $\pi n + \frac{\pi}{2}$ and $\pi(n\!+\!1)-\gamma_i$, $n\geq 0$, with corresponding multiplicities. Thus, the claim immediately follows from Lemma~\ref{RoucheMatrixLemm}.
\end{proof}

\section{Proof of the uniqueness theorem}
\setcounter{equation}{0}
\label{sec:UniquenessPrf}

Similarly to~\cite{CKIsoSets} we begin the proof of Theorem~\ref{uniqueness} with the following claim: the sequence $\{(\lambda_\alpha,\Pm_\alpha)\}_{\alpha\geq 0}$ uniquely determines the function $W(\lambda)=(\Bcondp\Fm)(\lambda)$ and the whole set of spectral data $\{(\lambda_\alpha,\Pm_\alpha,g_\alpha)\}_{\alpha\geq 0}$ uniquely determines the function $(\Bcondpd\Fm)(\lambda)$ and the Weyl-Titchmarsh function $m(\lambda)=-(\Bcondmd\Fp)(\lambda){[(\Bcondm\Fp)(\lambda)]}^{-1}$.

Let $L$ and $\wt L$ be two Sturm-Liouville operators (\ref{SLP})--(\ref{BoundConditions}) with some self-adjoint potentials $V,\wt V$ and possibly different operators $a,\wt{a}:H_-^\perp\to H_-^\perp$ and $b,\wt{b}:H_+^\perp\to H_+^\perp$ involved in the boundary conditions (\ref{BoundConditions}). We denote the spectral data of $L$ and $\wt{L}$ by $(\lambda_\alpha,\,\Pm_\alpha,\,{g}_\alpha)$ and $(\wt\lambda_\alpha,\,\wt\Pm_\alpha,\,\wt{g}_\alpha)$, respectively. In the same style we use the shorthands $\Gamma_\pm$, $\wt{\Gamma}_\pm$ which depend on $a$'s and $b$'s, and denote by $F_\pm$, $\wt F_\pm$ the corresponding fundamental solutions. It is worth noting that $\Gamma_\pm^\perp=\wt{\Gamma}_\pm^\perp$ by definition.


\begin{prop} \label{computations}
Let two vector-valued Sturm-Liouville operators $L$ and $\wt L$ be as above.

\smallskip

\noindent (i)\phantom{i} If $\lambda_\alpha=\wt{\lambda}_\alpha$ and
$\Pm_\alpha=\wt{\Pm}_\alpha$ for all $\alpha\geq0$, then
$(\Bcondp\Fm)(\lambda)=(\wt\Gamma_+\wt{F}_-)(\lambda)$ for all $\lambda\in\C$.

\smallskip

\noindent (ii) If, in addition,
$g_\alpha=\wt{g}_\alpha$ for all $\alpha\geq0$, then
$(\Bcondpd F_-)(\lambda)=(\Bcondpd \wt{F}_-)(\lambda)$ for all $\lambda\in\C$. Moreover, $b=\wt{b}$ and $F_-(1,\lambda)=\wt{F}_-(1,\lambda)$, $F'_-(1,\lambda)=\wt{F}'_-(1,\lambda)$ for all $\lambda\in\C$.
\end{prop}

We need the auxiliary lemma which mimics Lemma 2.1(ii) and Lemma 2.2 of \cite{CKIsoSets}:

\begin{lem} \label{auxiliary} \label{reverseW}
Let $V=V^*\in L^1([0,1];\C^{N\times N})$. Then

\smallskip

\noindent (i) for all $\alpha\geq 0$ one has
$G_\alpha=-\Pm_\alpha [(\Bcondp \Dot\Fm)(\lambda_\alpha)]^* (\Bcondpd\Fm)(\lambda_\alpha) \Pm_\alpha$\,;
\smallskip

\noindent (ii) the roots of the entire function $\det W(\lambda)$ are $\{\lambda_\alpha\}_{\alpha\geq 0}$, the multiplicity of each root is $k_\alpha$, and the following asymptotics are fulfilled as $\lambda\to\lambda_\alpha$:
\begin{equation} \label{asW-1}
[W(\lambda)]^{-1}=[(\lambda-\lambda_\alpha)^{-1}\Pm_\alpha+\Pm_\alpha^\perp][Z_\alpha^{-1}+O(\lambda-\lambda_\alpha)],
\end{equation}
where $\Pm_\alpha^\perp=I-\Pm_\alpha$, $Z_\alpha=\Dot W(\lambda_\alpha)\Pm_\alpha+W(\lambda_\alpha)\Pm_\alpha^\perp$ and $\det Z_\alpha\ne 0$ for all $\alpha\geq 0$.
\end{lem}
\begin{proof} See Appendix.
\end{proof}

\begin{proof}[Proof of Proposition~\ref{computations}]
\noindent(i) As $(\Gamma_+F_-)(\lambda_\alpha)\Pm_\alpha=0$ and $\Pm_\alpha=\wt{\Pm}_\alpha$ for all \mbox{$\alpha\geq0$},
 Proposition \ref{DivisionByWronsk}(i) implies that $(\Gamma_+F_-)(\lambda)[(\wt{\Gamma}_+\wt{F}_-)(\lambda)]^{-1}$
 is an entire function. Using asymptotics (\ref{WAsympthotic}),(\ref{InverseW}) and identities (\ref{Wnol}),(\ref{WnolInverse}) it is easy to conclude that
\begin{equation*}
(\Gamma_+F_-)(\lambda)[(\wt{\Gamma}_+\wt{F}_-)(\lambda)]^{-1}=W(\lambda)[\wt{W}(\lambda)]^{-1}=
W_0(\lambda)[W_0(\lambda)]^{-1}+O(\lambda^{-\frac{1}{2}})=I+O(\lambda^{-\frac{1}{2}})
\end{equation*}
for $\lambda\in C_n^{(\beta)}=\{\lambda:|\lambda|=(\pi n+\beta)^2\}$ and $\beta \ne 0,\pm\gamma_1,\dots,\pm\gamma_k,\frac{1}{2}\pi\mod\pi$. According to the Liouville theorem, this means $(\Gamma_+F_-)(\lambda)=(\wt{\Gamma}_+\wt{F}_-)(\lambda)$. As a byproduct, let us note that asymptotics (\ref{asymptoticsFminus}) imply
\begin{equation*} \begin{split}
0&=(\Gamma_+F_-)(\lambda)-(\wt{\Gamma}_+\wt{F}_-)(\lambda)\\
&= \begin{pmatrix} T_+^\perp & \lambda^{-\frac{1}{2}}T_+\end{pmatrix}\left[\begin{pmatrix} (A+B+U)\cos\sqrt{\lambda} & (B+U)\sin\sqrt{\lambda} \\ -(A+U)\sin\sqrt{\lambda} & U\cos\sqrt{\lambda}\end{pmatrix} +o(e^{|\im\,\sqrt{\lambda}|})
\right]\begin{pmatrix} T_-^\perp \\ \lambda^{-\frac{1}{2}}T_-\end{pmatrix}
\end{split}\end{equation*}
as $\lambda\to\infty$, where $A=a-\wt{a}$, $B=b-\wt{b}$ and $U=\frac{1}{2}\int_0^1(V(x)-\wt{V}(x))dx$. In particular, the leading terms should vanish, i.e.,
\begin{equation}\label{ABU=0}
\begin{array}{rr}
T_+^\perp(A+B+U)T_-^\perp=0, & T_+^\perp(B+U)T_-=0, \\
T_+(A+U)T_-^\perp=0, & T_+UT_-=0.
\end{array}
\end{equation}

\smallskip
\noindent(ii) From now on we use the notation $W(\lambda)=(\Gamma_+F_-)(\lambda)=(\wt{\Gamma}_+\wt{F}_-)(\lambda)$. Recall that $\Pm_\alpha=\wt{\Pm}_\alpha$ for all $\alpha\geq 0$.
Our intention is to apply Proposition \ref{DivisionByWronsk}(i) to the function $(\Bcondpd F_-)(\lambda)-(\Bcondpd \wt{F}_-)(\lambda)$ so we need to check that $(\Bcondpd F_-)(\lambda_\alpha)\Pm_\alpha=(\Bcondpd \wt{F}_-)(\lambda_\alpha)\wt{\Pm}_\alpha$ for all $\alpha\geq 0$. Using Lemma \ref{auxiliary} one obtains (for both $(V;a,b)$ and $(\wt{V};\wt{a},\wt{b})$)
\begin{equation*} \begin{split} 
 Z_\alpha^*(\Bcondpd \Fm)(\lambda_\alpha) \Pm_\alpha & = \Pm_\alpha[(\Bcondp\Dot{\Fm})(\lambda_\alpha)]^*(\Bcondpd\Fm)(\lambda_\alpha)\Pm_\alpha + \Pm_\alpha^\perp[(\Bcondp\Fm)(\lambda_\alpha)]^* (\Bcondpd\Fm)(\lambda_\alpha) \Pm_\alpha \\
 & = -G_\alpha + \Pm_\alpha^\perp[(\Bcondpd\Fm)(\lambda_\alpha)]^* (\Bcondp\Fm)(\lambda_\alpha) \Pm_\alpha \\ &= -G_\alpha,
\end{split} \end{equation*}
where we have used the identities
\[
[(\Bcondpd\Fm)(\lambda_\alpha)]^* (\Bcondp\Fm)(\lambda_\alpha)-[(\Bcondp\Fm)(\lambda_\alpha)]^* (\Bcondpd\Fm)(\lambda_\alpha) 
=\{\Fm,\Fm\}(\lambda_\alpha)=0
\]
(which follows from the fact that the Wronskian $\{\Fm,\Fm\}$ does not depend on~$x$, see also~(\ref{WronskianBcond})) and $(\Bcondp\Fm)(\lambda_\alpha) \Pm_\alpha=0$ (which follows from the definition of $\Pm_\alpha$). Since $Z_\alpha=\wt{Z}_\alpha$, we have
\[
((\Bcondpd F_-)(\lambda_\alpha)-(\Bcondpd \wt{F}_-)(\lambda_\alpha))\Pm_\alpha= (Z_\alpha^*)^{-1}(G_\alpha-\wt{G}_\alpha)=0\,.
\]
Therefore, the function $((\Bcondpd F_-)(\lambda)-(\Bcondpd \wt{F}_-)(\lambda))[W(\lambda)]^{-1}$ is entire. Further, it follows from (\ref{asymptoticsFminus}) that
\[
(\Bcondpd F_-)(\lambda)-(\Bcondpd \wt{F}_-)(\lambda) = \qForm{\Tp^\perp}{\Tp}{O(\lambda^{-\frac{1}{2}})}{O(\lambda^{-1})}
{o(1)}{o(\lambda^{-\frac{1}{2}})}{\Tm^\perp}{\Tm}\cdot e^{|\im\sqrt{\lambda}|}
\]
as $\lambda\to\infty$, where we have used the identities $T_+(A+U)T_-^\perp=0$ and $T_+UT_-=0$ in the second row, see (\ref{ABU=0}). Therefore, (\ref{WnolInverse}) and (\ref{InverseW}) imply
\[
((\Bcondpd F_-)(\lambda)-(\Bcondpd \wt{F}_-)(\lambda))[W(\lambda)]^{-1}=o(1),\qquad \lambda\in C_n^{(\beta)},
\]
where the contours $C_n^{(\beta)}$ are chosen as above and $n\to\infty$. According to the Liouville theorem, this means $(\Bcondpd F_-)(\lambda)=(\Bcondpd \wt{F}_-)(\lambda)$.

In order to prove that $b=\wt{b}$ we use asymptotics (\ref{asymptoticsFminus}) once more and arrive at
\begin{equation*} \begin{split}
0&=(\Gamma_+^\perp F_-)(\lambda)-({\Gamma}_+^\perp\wt{F}_-)(\lambda)\\
&= \begin{pmatrix} \lambda^{-\frac{1}{2}}T_+^\perp & T_+\end{pmatrix}\left[\begin{pmatrix} (A+U)\sin\sqrt{\lambda} & -U\cos\sqrt{\lambda} \\ (A+U)\cos\sqrt{\lambda} & U\sin\sqrt{\lambda}\end{pmatrix} +o(e^{|\im\,\sqrt{\lambda}|})
\right]\begin{pmatrix} T_-^\perp \\ \lambda^{-\frac{1}{2}}T_-\end{pmatrix}
\end{split}\end{equation*}
as $\lambda\to\infty$. Again, the leading terms should vanish, i.e.,
\begin{equation}\label{AU=0}
\begin{array}{rr}
T_+^\perp(A+U)T_-^\perp=0, & T_+^\perp UT_-=0, \\
T_+(A+U)T_-^\perp=0, & T_+UT_-=0.
\end{array}
\end{equation}
Subtracting the first rows of (\ref{ABU=0}) and (\ref{AU=0}) we conclude that $T_+^\perp B=b-\wt{b}=0$. Finally, one has
\begin{equation*} \begin{split}
F_-(1,\lambda)& =T_+^\perp(\Bcondpd F_-)(\lambda)-T_+(\Bcondp F_-)(\lambda)=\wt{F}_-(1,\lambda),\\
F'_-(1,\lambda)&=T_+(\Bcondpd F_-)(\lambda)+T_+^\perp((\Bcondp F_-)(\lambda)-bF_-(1,\lambda))=\wt{F}_-(1,\lambda).\qedhere
\end{split} \end{equation*}
\end{proof}

We are now in a position to prove the uniqueness theorem.

\begin{proof}[Proof of Theorem~\ref{uniqueness}]
The proof mimics the proof of Theorem~1.2 in~\cite{CKIsoSets}. Recall that Proposition~\ref{computations}(i) gives $(\Gamma_+F_-)(\lambda)=(\wt{\Gamma}_+\wt{F}_-)(\lambda)=W(\lambda)$ for all $\lambda\in\C$. Denote
\begin{equation*}\begin{split}
K(x,\lambda)&=\begin{pmatrix} \Fm(x,\lambda)& \Fp(x,\lambda)\\
\Fm'(x,\lambda) &\Fp'(x,\lambda) \end{pmatrix}\begin{pmatrix}
\wt{F}_-(x,\lambda)& \wt{F}_+(x,\lambda)\\ \wt{F}'_-(x,\lambda)
& \wt{F}'_+(x,\lambda) \end{pmatrix}^{-1}\\
& = -\begin{pmatrix} \Fm& \Fp\\\Fm' &\Fp' \end{pmatrix} (x,\lambda)
\begin{pmatrix}[W(\lambda)]^{-1} &0\\ 0 &[W^*(\ol\lambda)]^{-1} \end{pmatrix}
\begin{pmatrix}(\wt{F}'_+)^* &- \wt{F}_+^*\\-(\wt{F}'_-)^*& \wt{F}_-^* \end{pmatrix}(x,\ol\lambda).
\end{split}\end{equation*}
It is easy to see that the function $K(x,\lambda)$ satisfies the differential equation
\begin{equation}  \label{DiffEq}
K'(x,\lambda)=\begin{pmatrix}0 &I\\ V-\lambda I &0 \end{pmatrix}K(x,\lambda) -
K(x,\lambda)
\begin{pmatrix}0 &I\\ \widetilde V-\lambda I &0 \end{pmatrix}.
\end{equation}
Moreover, $K(1,\lambda)=I$ for all $\lambda\in\C$ due to Proposition~\ref{computations}(ii). Hence, $K(x,\lambda)$ is an entire function of $\lambda$ for each $x\in [0,1]$. Taking leading terms in (\ref{asymptoticsFminus}),(\ref{asymptoticsFplus}) and using Lemma~\ref{WnolInverseLemma} and (\ref{InverseW}) we conclude that
\begin{equation}
K(x,\lambda)=\begin{pmatrix}I+O(\lambda^{-\frac{1}{2}}) & O(\lambda^{-1})\\
O(1) & I+ O(\lambda^{-\frac{1}{2}}) \end{pmatrix},\qquad \lambda\in C_n^{(\beta)},
\end{equation}
as $n\to\infty$, where $C_n^{(\beta)}=\{\lambda:|\lambda|=(\pi n+\beta)^2\}$ and $\beta\ne 0,\pm\gamma_1,\dots,\pm\gamma_k,\frac{1}{2}\pi\mod \pi$. Hence,  $K(x,\lambda)$ is a constant matrix with the first row $(I~~0)$. In particular, this gives $F_\pm(x,\lambda)=\wt{F}_\pm(x,\lambda)$ for all $x\in [0,1]$ and $\lambda\in\C$ which implies $(V;a,b)=(\widetilde{V};\wt{a},\wt{b})$.
\end{proof}

\renewcommand{\thesection}{A}
\section{Appendix}
\setcounter{equation}{0}
\renewcommand{\theequation}{A.\arabic{equation}}
In this Appendix we collect several basic facts about zeroes of the Wronskian $W(\lambda)$ and residues of the Weyl-Titchmarsh function $m(\lambda)$, see~(\ref{MFunctDef}). In particular, we prove Lemma~\ref{auxiliary}, Proposition~\ref{DivisionByWronsk}(i) (recall that its second part has been derived from the first in Sect.~\ref{subsec:W-1asymp}) and Proposition~\ref{resMfunction}. The proofs essentially mimic the paper~\cite{CKIsoSets} devoted to vector-valued Sturm-Liouville operators with Dirichlet boundary conditions.

For two matrix-valued functions $F=F(x,\lambda)$ and $G=G(x,\lambda)$ we define their \emph{Wronskian} $\{F,G\}$ as follows:
\[
\{F,G\}(x,\lambda)= [{F}(x,\ol{\lambda})]^*G'(x,\lambda)- [F'(x,\ol{\lambda})]^*G(x,\lambda).
\]
If both $F,G$ are solutions of (\ref{DiffEquation}) with the same potential $V=V^*$, one can easily check that $\{F,G\}'=0$, and hence the Wronskian does not depend on $x\in [0,1]$. Moreover, in this case the values of $\{F,G\}$ at $x=1$ and $x=0$ can be rewritten as
\begin{equation}\begin{split}
\label{WronskianBcond}
\{F,G\}(\lambda) & = [(\Bcondpd F)(\overline{\lambda})]^* (\Bcondp G)(\lambda)-[(\Bcondp F)(\overline{\lambda})]^* (\Bcondpd G)(\lambda) \\
& = [(\Bcondmd F)(\overline{\lambda})]^* (\Bcondm G)(\lambda)-[(\Bcondm F)(\overline{\lambda})]^* (\Bcondmd G)(\lambda) \,.
\end{split}\end{equation}
In particular,
\[
W(\lambda)=\{\Fp,\Fm\}(\lambda)=(\Bcondp \Fm)(\lambda)= -[(\Bcondm\Fp)(\lambda)]^*.
\]

\begin{proof}[Proof of Lemma \ref{auxiliary}] (i) It is easy to check that $[\Fm^*\Fm](x,\lambda_\alpha)=-\{\Fm,\Dot\Fm\}(x,\lambda_\alpha)$. Therefore,
\[
G_\alpha=\Pm_\alpha[\{\Fm,\Dot\Fm\}(0,\lambda_\alpha)-\{\Fm,\Dot\Fm\}(1,\lambda_\alpha)] \Pm_\alpha = -\Pm_\alpha\{\Fm,\Dot\Fm\}(1,\lambda_\alpha)\Pm_\alpha
\]
as the initial conditions of $\Fm(x,\lambda)$ do not depend on $\lambda$. Moreover, the identity $\Bcondp [\Fm (\lambda_\alpha)\Pm_\alpha ]=0$ implies
\begin{equation*}\begin{split}
&\Fm(1,\lambda_\alpha)\Pm_\alpha=T_+^\perp \Fm(1,\lambda_\alpha)\Pm_\alpha,\\
&\Fm'(1,\lambda_\alpha)\Pm_\alpha= [T_+\Fm'-T_+^\perp b\Fm](1,\lambda_\alpha)\Pm_\alpha.
\end{split}\end{equation*}
Hence,
\begin{equation*}\begin{split}
G_\alpha & = -\Pm_\alpha[\Fm^*T_+^\perp\Dot{\Fm'}-((\Fm')^*T_+-\Fm^*bT_+^\perp)\Dot\Fm](1,\lambda_\alpha)\Pm_\alpha \\
& = -\Pm_\alpha [(\Fm^*T_+^\perp+(\Fm')^*T_+)(T_+^\perp(\Dot{\Fm'}+b\Dot{\Fm})-T_+\Dot{\Fm})](1,\lambda_\alpha)\Pm_\alpha \\
& = -\Pm_\alpha [(\Bcondpd \Fm)(\lambda_\alpha)]^*(\Bcondp \Dot{\Fm}) (\lambda_\alpha) \Pm_\alpha.
\end{split}\end{equation*}
This gives the result since $G_\alpha=G_\alpha^*$.

\smallskip

\noindent (ii) Note that $\det W(\lambda)=0$ if and only if there exists a nontrivial vector $h\in\C^N$ such that $W(\lambda)h=(\Bcondp\Fm)(\lambda)h=0$, i.e. if and only if $\lambda$ is an eigenvalue of the Sturm-Liouville operator $L$ given by (\ref{SLP})--(\ref{BoundConditions}). Since $W(\lambda_\alpha)\Pm_\alpha =0$, one has
\begin{equation} \begin{split}
W(\lambda)&=W(\lambda)\Pm_\alpha+W(\lambda)\Pm_\alpha^\perp\\
&= [(\lambda\!-\!\lambda_\alpha) \Dot W(\lambda_\alpha)+O((\lambda\!-\!\lambda_\alpha)^2)]
\Pm_\alpha+[W(\lambda_\alpha)+ O(\lambda-\lambda_\alpha)]\Pm_\alpha^\perp\\
&= [Z_\alpha +O(\lambda\!-\!\lambda_\alpha)][(\lambda\!-\!\lambda_\alpha)\Pm_\alpha+\Pm_\alpha^\perp]
\end{split} \notag \end{equation}
as $\lambda\to\lambda_\alpha$. If $\det Z_\alpha\ne 0$, this implies (\ref{asW-1}). Moreover, in this case the entire function $\det W(\lambda)$ has a root of multiplicity exactly $k_\alpha$ at $\lambda=\lambda_\alpha$ as
\[
\det W(\lambda)=[\det Z_\alpha+O(\lambda\!-\!\lambda_\alpha)]\cdot (\lambda-\lambda_\alpha)^{k_\alpha},\quad \lambda\to\lambda_\alpha.
\]
In order to prove that $\det Z_\alpha \ne 0$, let us assume that $Z_\alpha h=0$ for some $h\in \C^N$. Due to (i) and $(\Bcondp\Fm)(\lambda_\alpha)\Pm_\alpha=0$, we conclude that
\begin{equation*}\begin{split}
\langle g_\alpha \Pm_\alpha h\,,\Pm_\alpha h\rangle & =
\langle -(\Bcondp\Dot{\Fm})(\lambda_\alpha)\Pm_\alpha h \,,(\Bcondpd\Fm)(\lambda_\alpha)\Pm_\alpha h\rangle\\
& = \langle (\Bcondp{\Fm})(\lambda_\alpha)\Pm_\alpha^\perp h \,,(\Bcondpd\Fm)(\lambda_\alpha)\Pm_\alpha h\rangle \\
& = \langle (\Bcondpd{\Fm})(\lambda_\alpha)\Pm_\alpha^\perp h \,,(\Bcondp\Fm)(\lambda_\alpha)\Pm_\alpha h\rangle =0,
\end{split}\end{equation*}
where we have used the identity $\{\Fm,\Fm\}=0$ and~(\ref{WronskianBcond}). Therefore, $\Pm_\alpha h=0$ which further leads to $ W(\lambda_\alpha) \Pm_\alpha^\perp h = Z_\alpha h - \dot{W}(\lambda_\alpha)\Pm_\alpha h =0$. By definition $\Ker W(\lambda_\alpha)=\E_\alpha$, hence $\Pm_\alpha^\perp h =0$, i.e., $h=0$.
\end{proof}

\begin{proof}[Proof of Proposition~\ref{DivisionByWronsk}(i)] This is an easy corollary of Lemma~\ref{auxiliary}(ii). Indeed, if $\zeta(\lambda_\alpha)\Pm_\alpha=0$, then the function
\[
\zeta(\lambda)[\Bcondp\Fm(\lambda)]^{-1}= 
[\zeta(\lambda_\alpha)+O(\lambda\!-\!\lambda_\alpha)]\cdot[(\lambda-\lambda_\alpha)^{-1}\Pm_\alpha+\Pm_\alpha^\perp]\cdot O(1),\quad \lambda\to\lambda_\alpha,
\]
has no singularity at $\lambda_\alpha$.
\end{proof}

To prove Proposition~\ref{resMfunction} we need one more auxiliary result. Recall that we define the matrix-valued Weyl-Titchmarsh function of the Sturm-Liouville operator (\ref{SLP})--(\ref{BoundConditions}) as
\[
m(\lambda)= -(\Bcondmd\Fp)(\lambda){[(\Bcondm\Fp)(\lambda)]}^{-1}.
\]
Let $\Pm^\sharp_\alpha:\C^N\to \E_\alpha^\sharp$ be orthogonal projectors onto subspaces
\[
\E^\sharp_\alpha=\Ker\,(\Bcondm\Fp)(\lambda_\alpha) = \Ker\,[W^*(\lambda_\alpha)]\,.
\]

\begin{lem} \label{InitialCondMaps} The following holds true for all $\alpha\geq0$:

\smallskip
\noindent (i)\phantom{i} $\Ran [(\Bcondpd\Fm)(\lambda_\alpha)\Pm_\alpha] \subset\E^\sharp_\alpha$ and $\Ran [(\Bcondmd\Fp)(\lambda_\alpha)\Pm^\sharp_\alpha]\subset\E_\alpha$;

\smallskip
 \noindent (ii) $(\Bcondmd\Fp)(\lambda_\alpha)(\Bcondpd\Fm)(\lambda_\alpha)\Pm_\alpha=\Pm_\alpha$ and
 $(\Bcondpd\Fm)(\lambda_\alpha)(\Bcondmd\Fp)(\lambda_\alpha)\Pm^\sharp_\alpha=\Pm^\sharp_\alpha$.

\smallskip

\noindent In other words, $(\Bcondpd\Fm)(\lambda_\alpha)$ maps $\E_\alpha$ into $\E^\sharp_\alpha$,
$(\Bcondmd\Fp)(\lambda_\alpha)$ maps $\E^\sharp_\alpha$ into $\E_\alpha$, and these mappings are inverse to each other on $\E_\alpha$ and $\E^\sharp_\alpha$, respectively.
\end{lem}

\begin{proof}
Let
\[
\eta(x)=[\Fp(x,\lambda_\alpha)(\Bcondpd\Fm)(\lambda_\alpha) -\Fm(x,\lambda_\alpha)]\Pm_\alpha,\quad x\in [0,1].
\]
Clearly, $\eta(x)$ is a solution of the differential equation (\ref{DiffEquation}) and
\begin{equation*}\begin{split}
\Bcondp\eta & = -(\Bcondp\Fm)(\lambda_\alpha)\Pm_\alpha=0,\\
\Bcondpd\eta & = [(\Bcondpd\Fm)(\lambda_\alpha)-(\Bcondpd\Fm)(\lambda_\alpha)]\Pm_\alpha=0.
\end{split}\end{equation*}
Hence, $\eta(x)=0$ for all $x\in [0,1]$. In particular,
\begin{equation*}\begin{split}
\Bcondm\eta & = (\Bcondm\Fp)(\lambda_\alpha) (\Bcondpd\Fm)(\lambda_\alpha)\Pm_\alpha =0,\\
\Bcondmd\eta & = [(\Bcondmd\Fp)(\lambda_\alpha) (\Bcondpd\Fm)(\lambda_\alpha)-I]\Pm_\alpha=0.
\end{split}\end{equation*}
The first equation means $\Ran\,[(\Bcondpd\Fm)(\lambda_\alpha)\Pm_\alpha]\subset \Ker\,(\Bcondm\Fp)(\lambda_\alpha)=\E^\sharp_\alpha$ while the second gives the first part of (ii). Two other claims follow from similar arguments applied to the function $\eta^\sharp(x)=[\Fm(x,\lambda_\alpha)(\Bcondmd\Fp)(\lambda_\alpha) -\Fp(x,\lambda_\alpha)]\Pm^\sharp_\alpha$.
\end{proof}

\begin{proof}[Proof of Proposition~\ref{resMfunction}]
The identity $\{\Fp,\Fp\}=0$ (see~(\ref{WronskianBcond})) gives $m(\lambda)=[m(\overline{\lambda})]^*$. Moreover, (\ref{asW-1}) yields
\[
R_\alpha = \res\nolimits_{\lambda=\lambda_\alpha}m(\lambda)=-(\Bcondmd\Fp)(\lambda_\alpha)\cdot [Z_\alpha^{-1}]^*\Pm_\alpha\,.
\]
In particular, $R_\alpha^{\phantom{1}}\Pm^\perp_\alpha=0$. Thus, it is sufficient to prove that
$R_\alpha G_\alpha=\Pm_\alpha$. It follows from Lemma~\ref{auxiliary}(ii), the identity $[\Dot{W}(\lambda_\alpha)\Pm_\alpha]^*=[Z_\alpha\Pm_\alpha]^*$, and Lemma~\ref{InitialCondMaps}(i) that
\begin{equation*}\begin{split}
R_\alpha G_\alpha & =(\Bcondmd\Fp)(\lambda_\alpha)[Z_\alpha^{-1}]^*\cdot \Pm_\alpha[\Dot{W}(\lambda_\alpha)]^*\cdot (\Bcondpd\Fm)(\lambda_\alpha)\Pm_\alpha\\
& = (\Bcondmd\Fp)(\lambda_\alpha)[Z_\alpha^{-1}]^*\cdot \Pm_\alpha Z_\alpha^*\cdot \Pm_\alpha^\sharp(\Bcondpd\Fm)(\lambda_\alpha)\Pm_\alpha\,.
\end{split}\end{equation*}
Note that $Z_\alpha^*\Pm_\alpha^\sharp = \Pm_\alpha[\Dot{W}(\lambda_\alpha)]^*\Pm_\alpha^\sharp = \Pm_\alpha Z_\alpha^*\Pm_\alpha^\sharp$. Finally, Lemma~\ref{InitialCondMaps}(ii) implies
\begin{equation*}\begin{split}
R_\alpha G_\alpha & = (\Bcondmd\Fp)(\lambda_\alpha)\cdot [Z_\alpha^{-1}]^*Z_\alpha^*\Pm_\alpha^\sharp\cdot(\Bcondpd\Fm)(\lambda_\alpha)\Pm_\alpha \\
& =(\Bcondmd\Fp)(\lambda_\alpha)\Pm_\alpha^\sharp(\Bcondpd\Fm)(\lambda_\alpha)\Pm_\alpha  = \Pm_\alpha
\end{split}\end{equation*}
which gives the result.
\end{proof}


\def\cprime{$'$} \def\cprime{$'$} \def\cprime{$'$}
\providecommand{\bysame}{\leavevmode\hbox to3em{\hrulefill}\thinspace}
\providecommand{\MR}{\relax\ifhmode\unskip\space\fi MR }
\providecommand{\MRhref}[2]{%
  \href{http://www.ams.org/mathscinet-getitem?mr=#1}{#2}
}
\providecommand{\href}[2]{#2}

\end{document}